\documentclass[12pt]{article}
\usepackage{fullpage}
\usepackage{epsf,epsfig,amsfonts,amsgen,amsmath,amstext,amsbsy,amsopn,amsthm,amssymb, epstopdf,tikz,mathrsfs}
\usepackage{ebezier,eepic,mathtools,dsfont}
\usepackage{color,colordvi}
\usepackage{multirow}
\usepackage{url}
\usepackage{graphicx}
\usepackage{epsf}
\usepackage[format=hang, margin=10pt]{caption}
\newtheorem{theorem}{Theorem}[section]

\newtheorem{conjecture}[theorem]{Conjecture}
\newtheorem{lemma}[theorem]{Lemma}

\theoremstyle{definition}

\begin{document}

\title{On a conjecture concerning the property of chromatic polynomials with negative variables}

\author{Yan Yang\thanks{School of Mathematics and KL-AAGDM, Tianjin University, Tianjin, China: yanyang@tju.edu.cn.}\,
}

\date{}

\maketitle

\begin{abstract}
  Let \( G \) be a graph of order \( n \) and $P(G,x)$ be the chromatic polynomial of $G$.
Dong, Ge, Gong, Ning, Ouyang, and Tay (\emph{J. Graph Theory} {\bf96} (2021) 343) conjectured that
\(
\frac{d^k}{dx^k} \bigl( \ln[(-1)^n P(G, x)] \bigr) < 0
\)
holds for all \( k \geq 2 \) and \( x \in (-\infty, 0) \). We prove this conjecture for all \( k \geq 2 \) and \( x\leq -6.66\Delta k \), in which $\Delta$ is
the maximum degree of $G$.


\end{abstract}

\section{Introduction}
All graphs considered in this paper are simple and connected. For any graph $G=(V, E)$ and any positive integer $x$,  a {\it proper $x$-coloring} of $G$ is a mapping $c: V\rightarrow \{1,\ldots,x\}$ such that
$c(u)\neq c(v)$ whenever $uv\in E$. The {\it chromatic number} $\chi(G)$ is the least $x$ such that $G$ has a proper $x$-coloring.
In 1912, Birkhoff \cite{B1912}
introduced the {\it chromatic polynomial} $P(G,x)$ of $G$  to count the number of distinct proper $x$-colorings of $G$, i.e.,
$$P(G,x)=|\{c: V\rightarrow \{1,\ldots,x\} \mid c\text{ is proper}\}|.$$
It was hoped that the chromatic polynomial could be used to solve the four color conjecture of that time, because there is a clear relationship between chromatic polynomials and chromatic numbers, i.e., $\chi(G)$ is the smallest positive integer such that $P(G, \chi(G))\neq 0$.
Although this approach for the four color problem has not been successful so far, chromatic polynomials attracted widespread attention and have become an active topic nowadays. For a book devoted to the topic, see Dong, Koh and Teo~\cite{DKT05}.

By the principle of inclusion and exclusion, we have
\begin{equation*}
P(G,x) = \sum_{E' \subset E} (-1)^{|E'|}x^{\kappa(E')},
\end{equation*}
where \( \kappa(E') \) denotes the number of connected components in \( E' \).
It is clear that $P(G, x)$ is a polynomial in $x$. In 1932, Whitney \cite{W1932} interpreted the coefficients of \( P(G, x) \)
by introducing the notion of broken cycles. Let \( \eta: E \to \{1, 2, \dots, |E|\} \) be a bijection, for any cycle \( C \) in \( G \), the path \( C - e \) is called a {\it broken cycle} of \( G \) if \( e \) is the edge on \( C \) with \( \eta(e) \leq \eta(e') \) for every edge \( e' \) on \( C \).
\begin{theorem}[\cite{W1932}]\label{a}
Let \( G \) be a graph of order \( n \) and \( \eta: E\to \{1, 2, \dots, |E|\} \) be a bijection. Then,
\[
P(G, x) = \sum_{i=1}^{n} (-1)^{n-i} a_i(G) x^i,
\]
where \( a_i(G) \) is the number of spanning subgraphs of \( G \) with \( n - i \) edges which do not contain broken cycles.
\end{theorem}

Let $x$ be an arbitrary complex number in  $P(G,x)$, and $\Delta(G)$ (or simply $\Delta$) be the maximum degree of a graph $G$.
The distribution of the roots of $P(G,x)=0$ and  evaluations of $P(G,x)$ at some specific points are two of the many research problems related to chromatic polynomials. We now introduce the results concerning the bound on the roots of $P(G,x)=0$ that are needed in this note.

\begin{theorem}[\cite{S01}]\label{1}
 There exists a constant $K$ such that for every graph $G$, $P(G,x) \neq 0$ for all $x \in \mathbb{C}$ with $|x| > K \Delta$.
\end{theorem}
Sokal \cite{S01} showed the constant $K\leq 7.97$, Fern\'{a}ndez and Procacci \cite{FP08} improved to $K\leq 6.91$,
Jenssen, Patel, Regts \cite{J24} further improved to $K \leq 5.94$. Very recently, Bencs and Regts \cite{BR26} established a better bound $K \leq 4.25$.


Denote \( P^{(i)}(G, x) \) the \( i \)-th derivative of \( P(G, x) \). Recently, Bernardi and Nadeau \cite{BN20} gave a nice combinatorial interpretation of \( P^{(i)}(G, -j) \) for any nonnegative integers \( i \) and \( j \) in terms of acyclic orientations. When \( i = 0 \) and \( j = 0 \),  their result recovers the classical interpretations due to Stanley \cite{S1973} and to Greene and Zaslavsky \cite{GZ83} respectively. In this note, we study the property of $(\ln[(-1)^n P(G, x)])^{(k)}$ for $k\geq 1$ and $x<0$.

In 2020, Dong, Ge, Gong, Ning, Ouyang, and Tay \cite{D21} studied a parameter called the mean size of a broken-cycle-free spanning subgraph of a graph $G$ of order $n$, that is
$\epsilon(G)=\sum\limits_{i=1}^{n} (n-i) a_i(G)/\sum\limits_{i=1}^{n} a_i(G)$, where \( a_i(G) \) is the number as defined in Theorem \ref{a}. They proved that $\epsilon(T_n)<\epsilon(G)<\epsilon(K_n)$ holds for any graph $G$ of order $n$ which is neither the
complete graph $K_n$ nor a tree $T_n$ of order $n$. In their proof, they introduced a function $\epsilon(G,x)=P'(G,x)/P(G,x)$. One can deduce that $\epsilon(G)=n+\epsilon(G,-1)$ holds for any graph $G$ of order $n$. Thus, for any graphs $G$ and $H$ of the same order, $\epsilon(G)<\epsilon(H)$ is equivalent to $\epsilon(G,-1)<\epsilon(H,-1)$. In \cite{D21}, the authors proved that for any graph $G$ of order $n$ which is neither the complete graph nor a tree, $\epsilon(T_n,x)<\epsilon(G,x)<\epsilon(K_n,x)$ holds for all $x<0$. Then $\epsilon(T_n)<\epsilon(G)<\epsilon(K_n)$ follows.

It is clear that for any graph \( G \) of order \( n \),
\begin{equation*}
\epsilon(G,x)=\frac{P'(G, x)}{P(G, x)}=\frac{d}{dx} (\ln[(-1)^n P(G, x)])< 0
\end{equation*}
holds for all \( x < 0 \). Because for $x<0$, \((-1)^n P(G, x)>0\) and \((-1)^n P'(G, x)<0\). In \cite{D21}, the authors conjectured that this property holds for higher derivatives of the function \(\ln[(-1)^n P(G, x)]\) for $x<0$.

\begin{conjecture}[\cite{D21}]\label{con} Let \( G \) be a graph of order \( n \). Then
\[
\frac{d^k}{dx^k} (\ln[(-1)^n P(G, x)]) < 0
\]
holds for all \( k \geq 2 \) and \( x \in (-\infty, 0) \).
\end{conjecture}


We focus on Conjecture \ref{con} in this note, and prove the conjecture for all $k\geq 2$ and \( x\leq -6.66\Delta k \).  For $k=2$ or $3$, we have the following results.

\begin{lemma}\label{le}
Let \( G \) be a graph of order \( n \) and $K$ be the constant in Theorem \ref{1}. Then
\[
\frac{d^2}{dx^2} (\ln[(-1)^n P(G, x)]) < 0
\]
holds for \( x <-2K\Delta\), and  \[
\frac{d^3}{dx^3} (\ln[(-1)^n P(G, x)]) < 0
\]
holds for \( x <-(\sqrt{3}+1)K\Delta\).
\end{lemma}

\begin{proof} Suppose that the real roots of $P(G,x)=0$ are $x_1,\ldots,x_i$ and the complex roots are $a_1\pm ib_1,\ldots, a_j\pm ib_j$, then $i+2j=n$, and
$$P(G,x)=(x-x_1)\cdots(x-x_i)(x^2-2a_1x+a_1^2+b_1^2)\cdots(x^2-2a_jx+a_j^2+b_j^2).$$
By computing, we have \[
\frac{d^2}{dx^2} (\ln[(-1)^n P(G, x)])=\sum_{t=1}^i\frac{-1}{(x-x_t)^2}+\sum_{t=1}^j\frac{2b_t^2-2(x-a_t)^2}{\big((x-a_t)^2+b_t^2\big)^2}
\]
and
\[
\frac{d^3}{dx^3} (\ln[(-1)^n P(G, x)])=\sum_{t=1}^i\frac{2}{(x-x_t)^3}+\sum_{t=1}^j\frac{4(x-a_t)\big((x-a_t)^2-3b_t^2\big)}{\big((x-a_t)^2+b_t^2\big)^3}.
\]
It is well known that the $P(G,x)=0$ has no negative real roots. Hence, for $x<0$, $\frac{-1}{(x-x_t)^2}<0$ and $\frac{2}{(x-x_t)^3}<0$.
When $x<-(|a_t|+|b_t|)$, $\frac{2b_t^2-2(x-a_t)^2}{\big((x-a_t)^2+b_t^2\big)^2}<0$; when $x<-(|a_t|+|\sqrt{3}b_t|)$, $\frac{4(x-a_t)\big((x-a_t)^2-3b_t^2\big)}{\big((x-a_t)^2+b_t^2\big)^3}<0$.

From Theorem \ref{1}, $|a_t+ib_t|\leq K\Delta$.  Then \( x <-2K\Delta\) implies $x<-(|a_t|+|b_t|)$,  \( x <-(\sqrt{3}+1)K\Delta\) implies $x<-(|a_t|+|\sqrt{3}b_t|)$. The lemma follows.
\end{proof}

Combining Lemma \ref{le} and $K\leq 4.25$, we get that $(\ln[(-1)^n P(G, x)])'' < 0$ for $x<-8.5\Delta$, and $(\ln[(-1)^n P(G, x)])''' < 0$ for $x<-4.25(\sqrt{3}+1)\Delta$.
When $k$ becomes larger, the method in the above lemma is no longer efficient. In Section 2, we present a different approach to solve this problem.

\section{Main results}
For a graph $G$ of order $n$, let $\alpha_1,\ldots, \alpha_n$ be the roots of $P(G,x)=0$ and \[
c_i=-\dfrac{1}{i}\sum\limits_{j=1}^n \alpha_j^i.
\]
Fadnavis \cite{F15} deduced the exact values of $c_1$ and $c_2$.
\begin{lemma} [\cite{F15}] \label{4} Let $|T|$ be the number of triangles in $G$, then $c_1=-|E|$ and $c_2=-|T|-|E|/2.$
\end{lemma}

Let $K$ be the constant in Theorem \ref{1}. Using Taylor expansion, one can deduce the following result.
\begin{theorem} [\cite{F15}]\label{3} For a graph $G$ of order $n$ and $|x|>K\Delta$, we have
\[
\ln \frac{P(G,x)}{x^n} = \sum_{i=1}^\infty c_i x^{-i}.
\]
\end{theorem}

\begin{lemma}\label{se} For $x\leq q$, $q<-K\Delta$, we have
\\(i) $\sum\limits_{i=1}^\infty c_i x^{-i}$ converges uniformly.
\\(ii) $\sum\limits_{i=1}^\infty (c_i x^{-i})^{(k)}, ~~k=1,2,\ldots$, converges uniformly.
\\(iii) Both $\sum\limits_{\substack{ i =1\\ \text{$i$ is odd}}}^\infty \big((\frac{K\Delta}{x})^{i}\big)^{(k)}$ and $\sum\limits_{\substack{ i =1\\ \text{$i$ is even}}}^\infty \big((\frac{K\Delta}{x})^{i}\big)^{(k)}$,~~$k=1,2,\ldots$,  converge uniformly.
\end{lemma}

\begin{proof} The proofs of the three results are similar, so we only present the proof of conclusion (ii).
From the definition of $c_i$ and Theorem \ref{1}, we have
\(
|c_i| \leq \frac{1}{i} \sum\limits_{j=1}^n |\alpha_j|^i \leq \frac{n}{i}(K\Delta)^i
\).
Then, for $x\leq q$, $q<-K\Delta$, we have \begin{equation*}\label{e1}|(c_i x^{-i})^{(k)}|=|c_i \dfrac{(-1)^k i(i+1)\cdots(i+k-1)}{x^{i+k}}|\leq \dfrac{n}{i} (K\Delta)^i\dfrac{i(i+1)\cdots(i+k-1)}{(-q)^{i+k}}.\end{equation*}
And the positive term series $\sum\limits_{ i =1}^\infty \dfrac{(i+1)\cdots(i+k-1)}{(-q)^{k}}(\dfrac{K\Delta}{-q})^{i}$ converges.
From Weierstrass M-test, $\sum\limits_{i=1}^\infty (c_i x^{-i})^{(k)}$, $x\leq q$, $q<-K\Delta$ converges uniformly.
\end{proof}

From Lemma \ref{se} and properties of uniform convergence series, we have
\begin{equation}\label{e2}(\sum\limits_{i=1}^\infty c_i x^{-i})^{(k)}=(\sum\limits_{\substack{ i =1\\ \text{$i$ is odd}}}^\infty c_i x^{-i})^{(k)}+(\sum\limits_{\substack{ i =1\\ \text{$i$ is even}}}^\infty c_i x^{-i})^{(k)},\end{equation}
for $k\geq 0$ ($k=0$ refers to the series itself), and the term-by-term  differentiation is valid.

Because $K\leq 4.25$, \(|c_i| \leq \frac{n}{i}(K\Delta)^i\leq \frac{n}{i}(4.25\Delta)^i\). And $$(c_i x^{-i})^{(k)}=c_i \frac{(-1)^k i(i+1)\cdots(i+k-1)}{x^{i+k}}.$$
For $x\leq q$, $q<-4.25\Delta$ and $k\geq 2$, when $i$ is odd, $\frac{(-1)^k}{x^{i+k}}<0$,
$$(c_i x^{-i})^{(k)}\leq \dfrac{-n(4.25\Delta)^i(-1)^{k}(i+1)\cdots(i+k-1)}{x^{i+k}}=n(4.25\Delta)^{i}(\frac{1}{x^{i+1}})^{(k-1)};$$
when $i$ is even, $\frac{(-1)^k}{x^{i+k}}>0$,
$$(c_i x^{-i})^{(k)}\leq \dfrac{n(4.25\Delta)^i(-1)^{k}(i+1)\cdots(i+k-1)}{x^{i+k}}=-n(4.25\Delta)^{i}(\frac{1}{x^{i+1}})^{(k-1)}.$$
Hence
\begin{equation}\label{e3}(\sum\limits_{\substack{ i =1\\ \text{$i$ is odd}}}^\infty c_i x^{-i})^{(k)}\leq \frac{n}{4.25\Delta}(\sum\limits_{\substack{ i =1\\ \text{$i$ is odd}}}^\infty (\frac{4.25\Delta}{x})^{i+1})^{(k-1)}\end{equation}
and \begin{equation}\label{e4}(\sum\limits_{\substack{ i =1\\ \text{$i$ is even}}}^\infty c_i x^{-i})^{(k)}\leq -\frac{n}{4.25\Delta}(\sum\limits_{\substack{ i =1\\ \text{$i$ is even}}}^\infty (\frac{4.25\Delta}{x})^{i+1})^{(k-1)}\end{equation}

\begin{lemma}\label{2} For a graph $G$ of order $n$, $k\geq 1$ and $x<0$,
\[(\ln[(-1)^n P(G, x)])^{(k)}=(\ln\frac{P(G,x)}{x^n})^{(k)}+\frac{(k-1)!(-1)^{k-1}n}{x^{k}}. \]
\end{lemma}

\begin{proof} If $x<0$, then $\ln\frac{P(G,x)}{x^n}=\ln(-1)^nP(G,x)-\ln(-1)^nx^n$.
Because $(\ln(-1)^nx^n)^{(k)}=\frac{(k-1)!(-1)^{k-1}n}{x^{k}}$, the lemma follows.
\end{proof}

\begin{theorem} \label{main}For a graph $G$ of order $n$ and $k\geq 2$,
\[(\ln[(-1)^n P(G, x)])^{(k)}<0, \] holds for all $x\leq -6.66\Delta k$.\end{theorem}

\begin{proof} If $n=1$, then $P(G,x)=x$, $(\ln(-x))^{(k)}=\frac{(-1)^{k-1}(k-1)!}{x^k}<0$ holds for all $x<0$. We now assume $n\geq 2$.

From Theorem \ref{3} and Lemma \ref{2},
\begin{eqnarray}\label{e0}(\ln[(-1)^n P(G, x)])^{(k)}=(\sum\limits_{i=1}^\infty c_i x^{-i})^{(k)}+\frac{(k-1)!(-1)^{k-1}n}{x^{k}}.\end{eqnarray}
Now we consider $(\sum\limits_{i=1}^\infty c_i x^{-i})^{(k)}$. From \eqref{e2},
$$(\sum\limits_{i=1}^\infty c_i x^{-i})^{(k)}=(\frac{c_1}{x})^{(k)}+(\frac{c_2}{x^2})^{(k)}+(\sum\limits_{\substack{ i =3\\ \text{$i$ is odd}}}^\infty c_i x^{-i})^{(k)}+(\sum\limits_{\substack{ i =4\\ \text{$i$ is even}}}^\infty c_i x^{-i})^{(k)}.$$
From Lemma \ref{4}, $c_1=-|E|$ and $c_2=-|T|-|E|/2.$ By computing, \begin{eqnarray*}(\frac{c_1}{x})^{(k)}+(\frac{c_2}{x^2})^{(k)}&=&-|E| \frac{(-1)^{k} \cdot k!}{x^{k+1}} - (|T|+\frac{|E|}{2}) \frac{(-1)^{k} \cdot (k+1)!}{x^{k+2}}\end{eqnarray*}
Because $G$ is a connected graph with order $n\geq 2$, and the maximum degree is $\Delta$,  we have $\frac{n}{2}\leq |E|\leq\frac{\Delta n}{2}$. Then
\begin{eqnarray}\label{e5}(\frac{c_1}{x})^{(k)}+(\frac{c_2}{x^2})^{(k)}&\leq &-\frac{\Delta n}{2} \frac{(-1)^{k}k!}{x^{k+1}} - \frac{n}{4} \frac{(-1)^{k} (k+1)!}{x^{k+2}}.\end{eqnarray}
By \eqref{e3}\eqref{e4}, some computations and simplifications, we obtain
\begin{eqnarray}\label{e6}
&&(\sum\limits_{\substack{ i =3\\ \text{$i$ is odd}}}^\infty c_i x^{-i})^{(k)}+(\sum\limits_{\substack{ i =3\\ \text{$i$ is even}}}^\infty c_i x^{-i})^{(k)}\nonumber\\
&\leq& \frac{n}{4.25\Delta}(\sum_{\substack{i=3 \\ i \text{ is odd}}}^{\infty}(\frac{4.25\Delta}{x})^{i+1})^{(k-1)}-\frac{n}{4.25\Delta}(\sum_{\substack{i=4 \\ i \text{ is even}}}^{\infty}(\frac{4.25\Delta}{x})^{i+1})^{(k-1)}\nonumber\\
&=& \frac{n}{4.25\Delta}(\frac{(4.25\Delta)^4}{x^3(x+4.25\Delta)})^{(k-1)}=n(\frac{1}{x}-\frac{4.25\Delta}{x^2}+\frac{(4.25\Delta)^2}{x^3}-\frac{1}{x+4.25\Delta})^{(k-1)}\nonumber\\
&=&n(-1)^{k-1}\Big(\frac{(k-1)!}{x^{k}}-4.25\Delta \frac{k!}{x^{k+1}} + (4.25\Delta)^{2}\frac{(k+1)!}{2x^{k+2}}-\frac{(k-1)!}{(x+4.25\Delta)^{k}}\Big)
\end{eqnarray}

Combining \eqref{e0}, \eqref{e5} and \eqref{e6}, we have
\begin{eqnarray*}&&(\ln[(-1)^n P(G, x)])^{(k)}\\
&\leq&-\frac{\Delta n}{2} \frac{(-1)^{k}k!}{x^{k+1}} - \frac{n}{4} \frac{(-1)^{k} (k+1)!}{x^{k+2}} \\
&&+n(-1)^{k-1}\Big(\frac{2(k-1)!}{x^{k}}-4.25\Delta \frac{k!}{x^{k+1}} + (4.25\Delta)^{2}\frac{(k+1)!}{2x^{k+2}}-\frac{(k-1)!}{(x+4.25\Delta)^{k}}\Big)\\
&=&n(-1)^{k-1}(k-1)!\frac{(x+4.25\Delta)^{k}\big(8x^2-15x\Delta k+k(k+1)+2(4.25\Delta)^2k(k+1)\big)-4x^{k+2}}{4x^{k+2}(x+4.25\Delta)^{k}}
\end{eqnarray*}

To prove $(\ln[(-1)^n P(G, x)])^{(k)}<0$ for $x\leq -6.66\Delta k$, it is enough to verify the following inequality for $x\leq -6.66\Delta k$,
\begin{eqnarray}\label{e7}(-1)^{k-1}(x+4.25\Delta)^{k}\big(8x^2-15x\Delta k+k(k+1)+2(4.25\Delta)^2k(k+1)\big)<(-1)^{k-1}4x^{k+2}.\end{eqnarray}
We denote $f_1(x)=8x^2-15x\Delta k+k(k+1)+2(4.25\Delta)^2k(k+1)$. Because $f_1(x)>0$ and $(-1)^{k-1}x^{k}<0$,
inequality \eqref{e7} is equivalent to
$$\frac{(-1)^{k-1}(x+4.25\Delta)^{k}}{(-1)^{k-1}x^{k}}>\frac{4x^2}{f_1(x)},$$ that is
\begin{eqnarray*}\label{e8}(1+\frac{4.25\Delta}{x})^{k}>\frac{4x^2}{f_1(x)}.\end{eqnarray*}
By using Bernoulli's Inequality, we have $(1+\frac{4.25\Delta}{x})^{k}\geq 1+\frac{4.25\Delta k}{x},$  when $x<-4.25\Delta k$.
We let
$$1+\frac{4.25\Delta k}{x}> \frac{4x^2}{f_1(x)}$$ which is equivalent to $$xf_1(x)+4.25k\Delta f_1(x)-4x^3< 0.$$
Let $f(x)=xf_1(x)+4.25k\Delta f_1(x)-4x^3.$
We can check that when $x<-4.25\Delta k$, $f'(x)>0$, $f(x)$ is increasing strictly. By computing, we have $f(-6.66\Delta k)<0$.
Hence when $x\leq -6.66\Delta k$, $f(x)<0$, $(\ln[(-1)^n P(G, x)])^{(k)}<0$ follows.  The proof is complete.
\end{proof}

We note that to obtain inequalities \eqref{e3} and \eqref{e4} required for the proof of Theorem \ref{main}, we need to estimate the value of $c_i$.
In this note, we estimate $c_i$ by using $|c_i| \leq \frac{n}{i}(K\Delta)^i$. Thus, a new bound on $K$ gives a new bound on $c_i$, and a smaller
$K$ yields a better bound.


\section*{Acknowledgements}

The author thanks Guus Regts for informing her of reference \cite{BR26}, which provides the best bound on $K$ to date and is very helpful for the present note.
The author is supported by National Natural Science Foundation of China (No. 12371350).


\begin{thebibliography}{99}
\frenchspacing
\bibitem{BN20} O. Bernardi and P. Nadeau, Combinatorial reciprocity for the chromatic polynomial and the chromatic
symmetric function, \emph{Discrete Math.}, \textbf{343} (2020), 111989.

\bibitem{B1912} G. D. Birkhoff, A determinant formula for the number of ways of coloring a map, \emph{Ann. of Math.}, {\bf14}
(1912) 42--46.

\bibitem{BR26} F. Bencs and G. Regts, Improved bounds on the zeros of the chromatic polynomial of graphs and claw-free graphs,
https://arxiv.org/abs/2505.04366, (2026).


\bibitem{D21}  F. M. Dong,  J. Ge,  H. L. Gong,  B. Ning,  Z. Ouyang,  EG. Tay,  Proving a
conjecture on chromatic polynomials by counting the number of acyclic orientations,
\emph{J. Graph Theory}, \textbf{96} (2021), 343--360.

\bibitem{DKT05} F. M. Dong, K. M. Koh, and K. L. Teo, Chromatic Polynomials and Chromaticity of Graphs, World
Scientific, Singapore, 2005.

\bibitem{F15} S. Fadnavis, A note on the shameful conjecture, \emph{European J. Combin.}, \textbf{47} (2015), 115--122.

\bibitem{FP08} R. Fern\'{a}ndez and A. Procacci, Regions without complex zeros for chromatic polynomials
on graphs with bounded degree, \emph{Combin. Probab. Comput.}, \textbf{17(2)} (2008), 225--238.


\bibitem{GZ83}C. Greene and T. Zaslavsky, On the interpretation of Whitney numbers through arrangements of hyperplanes,
zonotopes, non-Radon partitions, and orientations of graphs, \emph{Trans. Amer. Math. Soc.}, \textbf{280} (1983), 97--124.25.

\bibitem{J24}M. Jenssen, V. Patel, G. Regts, Improved bounds for the zeros of the chromatic polynomial via Whitney's Broken Circuit Theorem, \emph{J. Combin. Theory, Ser. B}, \textbf{169} (2024), 233--252.

\bibitem{S1973}R. P. Stanley, Acyclic orientations of graphs, \emph{Discrete Math.}, \textbf{5} (1973), 171--178.

\bibitem{S01}A. D. Sokal, Bounds on the complex zeros of (di)chromatic polynomials and Potts-model partition
functions, \emph{Combin. Probab. Comput.}, \textbf{10(1)} (2001), 41--77.


\bibitem{W1932} H. Whitney, A logical expansion in mathematics, \emph{Bull. Amer. Math. Soc.}, \textbf{38} (1932), 572--579.


\end{thebibliography}
\end{document}